\documentclass{amsart}

\input xy
\xyoption{all}

\usepackage{amsmath}
\usepackage{hyperref}
\usepackage[utf8]{inputenc}
\usepackage{amssymb}
\usepackage[nospace,noadjust]{cite}
\usepackage{amsthm}  
\usepackage[final]{showkeys}

\newtheorem{same}{This should never appear}[section]
\newtheorem{defin}[same]{Definition}

\newtheorem{remark}[same]{Remark}
\newtheorem{theorem}[same]{Theorem}
\newtheorem{thm}[same]{Theorem}
\newtheorem{example}[same]{Example}
\newtheorem{lemma}[same]{Lemma}

\newtheorem{fact}[same]{Fact}
\newtheorem{question}[same]{Question}
\newtheorem{cor}[same]{Corollary}

\newtheorem{defin*}{Definition}

\newtheorem*{theorem*}{Theorem}
\newtheorem*{theorem**}{Theorem 1.2}

\newbox\noforkbox \newdimen\forklinewidth
\setbox0\hbox{$\textstyle\smile$}
\setbox1\hbox to \wd0{\hfil\vrule width \forklinewidth depth-2pt
 height 10pt \hfil}
\setbox\noforkbox\hbox{\lower 2pt\box1\lower 2pt\box0\relax}
\def\unionstick{\mathop{\copy\noforkbox}\limits}

\def\nonfork_#1{\unionstick_{\textstyle #1}}

\setbox0\hbox{$\textstyle\smile$}
\setbox1\hbox to \wd0{\hfil{\sl /\/}\hfil}
\setbox2\hbox to \wd0{\hfil\vrule height 10pt depth -2pt width
              \forklinewidth\hfil}
\newbox\doesforkbox
\setbox\doesforkbox\hbox{\lower 2pt\box1 \lower 2pt\box2\lower2pt\box0\relax}
\def\nunionstick{\mathop{\copy\doesforkbox}\limits}

\def\fork_#1{\nunionstick_{\textstyle #1}}

\newcommand{\Mod}{\textrm{Mod}}

\newcommand{\ba}{\mathbf{a}}
\newcommand{\bb}{\mathbf{b}}

\newcommand{\bx}{\mathbf{x}}

\newcommand{\rest}{\upharpoonright}

\newcommand{\K}{\mathbf{K}}
\newcommand{\LS}{\operatorname{LS}}
\def\lee{\preceq}
\newcommand{\Ll}{\mathbb{L}}

\newcommand{\leap}[1]{\le_{#1}}

\newcommand{\geap}[1]{\ge_{#1}}

\newcommand{\lea}{\leap{\K}}

\newcommand{\gea}{\geap{\K}}

\newcommand{\seq}[1]{\langle #1 \rangle}

\newcommand{\tp}{\operatorname{tp}}

\newcommand{\gtp}{\mathbf{tp}}
\newcommand{\gS}{\mathbf{S}}

\newcommand{\fct}[2]{{}^{#1}#2}

\newcommand{\Ii}{\mathbb{I}}
\newcommand{\muwd}{\mu_{\text{wd}}}
\newcommand{\PC}{\operatorname{PC}}

\title{Universal classes near $\aleph_1$}
\date{\today. }

\author{Marcos Mazari-Armida}
\email{mmazaria@andrew.cmu.edu}
\urladdr{http://www.math.cmu.edu/~mmazaria/ }
\address{Department of Mathematical Sciences \\ Carnegie Mellon University \\ Pittsburgh, Pennsylvania, USA}

\author{Sebastien Vasey}
\email{sebv@math.harvard.edu}
\urladdr{http://math.harvard.edu/\textasciitilde sebv/}
\address{Department of Mathematics \\ Harvard University \\ Cambridge, Massachusetts, USA}

\begin{document}

\begin{abstract}
  Shelah has provided sufficient conditions for an $\Ll_{\omega_1, \omega}$-sentence $\psi$ to have arbitrarily large models and for a Morley-like theorem to hold of $\psi$. These conditions involve structural and set-theoretic assumptions on all the $\aleph_n$'s. Using tools of Boney, Shelah, and the second author, we give assumptions on $\aleph_0$ and $\aleph_1$ which suffice when $\psi$ is restricted to be universal:

\begin{theorem*}
  Assume $2^{\aleph_{0}} < 2 ^{\aleph_{1}}$. Let $\psi$ be a universal $\Ll_{\omega_{1}, \omega}$-sentence.

  \begin{enumerate}
  \item If $\psi$ is categorical in $\aleph_{0}$ and $1 \leq \Ii(\psi, \aleph_{1}) < 2 ^{\aleph_{1}}$, then $\psi$ has arbitrarily large models and categoricity of $\psi$ in some uncountable cardinal implies categoricity of $\psi$ in all uncountable cardinals.
  \item If $\psi$ is categorical in $\aleph_1$, then $\psi$ is categorical in all uncountable cardinals.
  \end{enumerate}
\end{theorem*}

The theorem generalizes to the framework of $\Ll_{\omega_1, \omega}$-definable tame abstract elementary classes with primes.
\end{abstract}

\maketitle




\section{Introduction}
In a milestone paper, Shelah \cite{sh87a, sh87b} gives the following classification-theoretic analysis of $\Ll_{\omega_1, \omega}$-sentences:

\begin{fact}\label{sh}
  Assume that $2^{\aleph_n} < 2^{\aleph_{n + 1}}$ for all $n < \omega$. Let $\psi \in \Ll_{\omega_1, \omega}$ be a complete sentence. Assume that $\psi$ has an uncountable model and for all $n > 0$, $\Ii(\psi, \aleph_{n}) < \muwd(\aleph_{n})$.\footnote{See \cite[VII.0.4]{shelahaecbook} for a definition of $\muwd$ and \cite[VII.0.5]{shelahaecbook} for some of its properties. We always have that $2^{\aleph_n} \le \muwd (\aleph_{n + 1})$.} Then $\psi$ has arbitrarily large models and categoricity of $\psi$ in \emph{some} uncountable cardinal implies categoricity of $\psi$ in \emph{all} uncountable cardinals.
\end{fact}

It is provably necessary to make hypotheses on all the $\aleph_n$'s: a family of examples of Hart and Shelah \cite{hash323} (analyzed in detail by Baldwin and Kolesnikov \cite{untame}) gives for each $n < \omega$ an $\Ll_{\omega_1, \omega}$-sentence $\psi_n$ which is categorical in $\aleph_0$, $\aleph_1$, $\ldots$, $\aleph_n$ but not in any cardinal above $\aleph_n$.

In the present paper, we show that if we restrict the complexity of the sentence, then it suffices to make model-theoretic and set-theoretic assumptions on $\aleph_0$ and $\aleph_1$. More precisely:  

\textbf{Theorem \ref{main-theorem}.}
\textit{Assume $2^{\aleph_0} < 2^{\aleph_1}$. Let $\psi$ be a universal $\Ll_{\omega_1, \omega}$ sentence (i.e.\ $\psi$ is of the form $\forall \bx \phi (\bx)$, where $\phi$ is quantifier-free). If $\psi$ is categorical in $\aleph_0$ and $1\leq \Ii(\psi, \aleph_1) < 2^{\aleph_1}$, then:
\begin{enumerate}
\item $\psi$ has arbitrarily large models.
\item If $\psi$ is categorical in \emph{some} uncountable cardinal then $\psi$ is categorical in \emph{all} uncountable cardinals.
\end{enumerate}}

We more generally prove Theorem \ref{main-theorem} for universal classes (classes of models closed under isomorphisms, substructures, and unions of $\subseteq$-increasing chains, see Definition \ref{univ-def} and Fact \ref{tarski-fact}) in a countable vocabulary. The assumption of categoricity in $\aleph_0$ can be removed if we instead assume categoricity in $\aleph_1$. In this case, we obtain the following upward categoricity transfer:

\textbf{Theorem \ref{categ-upward}.}
\textit{Assume $2^{\aleph_0} < 2^{\aleph_1}$. Let $\psi$ be a universal $\Ll_{\omega_1, \omega}$ sentence. If $\psi$ is categorical in $\aleph_1$, then it is categorical in all uncountable cardinals.}

The statements of Theorems \ref{main-theorem} and \ref{categ-upward} should be compared to the second author's eventual categoricity theorem for universal classes\cite{vaseyf}.

\begin{fact}\label{univ-categ}
  Let $\psi$ be a universal $\Ll_{\omega_1, \omega}$-sentence. If $\psi$ is categorical in \emph{some} $\mu \ge \beth_{\beth_{\omega_1}}$, then $\psi$ is categorical in \emph{all} $\mu' \ge \beth_{\beth_{\omega_1}}$.
\end{fact}

Fact \ref{univ-categ} is a ZFC theorem while the results of this paper use $2^{\aleph_0} < 2^{\aleph_1}$. However, Fact \ref{univ-categ} is an eventual statement, valid for ``big'' cardinals (in fact there is a generalization to any universal class, not necessarily in a countable vocabulary), while the focus of this paper is on structural properties holding in $\aleph_0$ and $\aleph_1$.

The reader may wonder: are there any interesting examples of eventually categorical universal classes? After the initial submission of this paper, Hyttinen and Kangas \cite{univ-cat} showed that the answer is no: in any universal class categorical in a high-enough regular cardinal, any big-enough model will eventually look like either a set or a vector space (the methods are geometric in nature and also eventual, hence completely different from the tools used in this paper). Thus a reader wanting a non-trivial example illustrating e.g.\ Theorem \ref{categ-upward} is out of luck: the statement of Theorem \ref{categ-upward} combined with the Hyttinen-Kangas result implies that \emph{any} example will eventually look like a class of vector spaces or a class of sets!\footnote{It is however possible to add some noise in the low cardinals, see Example \ref{morley-example} here.} One can think of this result as saying that an eventual version of Zilber's trichotomy holds for universal classes (but since algebraically closed fields are not universal, it is really a dichotomy).

Nevertheless, we still believe that the theorems of this paper are important for several reasons. First, the fact that there are no nontrivial examples is itself not obvious and Theorem \ref{categ-upward} helps establish it. Second, it has many times been asked whether Morley's categoricity theorem can be applied to any interesting examples, and so far none has been found: the point is that the \emph{methods} used to prove Morley's theorem are important. Similarly, we believe that the methods to prove the theorems here (good frames and tameness) are important to develop a classification theory of AECs -- the statements of Theorems \ref{main-theorem} and \ref{categ-upward} here are showcases for the methods. Third, while Hyttinen and Kangas' proofs seem to only work for universal classes, our result can be generalized\footnote{Similarly, Fact \ref{univ-categ} can be generalized. See for example the recent result of Ackerman, Boney, and the second author on multiuniversal classes \cite{abv-categ-multi-v2}.} as follows:

\textbf{Theorem \ref{general-thm}.}
  Assume $2^{\aleph_0} < 2^{\aleph_1}$. Let $\K$ be an AEC with $\LS (\K) = \aleph_0$. Assume that $\K$ has primes, is $\aleph_0$-tame, and is $\PC_{\aleph_0}$ (see \cite[I.1.4]{shelahaecbook}, this is essentially the class of reducts of models of an $\Ll_{\omega_1,\omega}$-sentence).

  \begin{enumerate}
  \item If $\K$ is categorical in $\aleph_0$ and $1 \le \Ii (\K, \aleph_1) < 2^{\aleph_1}$, then $\K$ has arbitrarily large models and categoricity in \emph{some} uncountable cardinal implies categoricity in \emph{all} uncountable cardinals.
  \item If $\K$ is categorical in $\aleph_1$, then $\K$ is categorical in all uncountable cardinals.
  \end{enumerate}

The hypotheses of Theorem \ref{general-thm} are very general: they encompass for example the class of models of any $\aleph_0$-stable first-order theory (the setup of Morley's theorem) as well as any quasiminimal pregeometry class \cite{quasimin} (see e.g.\ \cite[4.2]{shvas-apal} for why they are $\PC_{\aleph_0}$). There are many such classes (e.g.\ the pseudoexponential fields \cite{zilber-pseudoexp}) which are not sets or vector spaces.

It is worth noting, that the results of this paper are direct consequences of putting together general facts about AECs (many only recently discovered): Shelah's construction of good frames \cite[\S II.3]{shelahaecbook}, Boney's proof of tameness in universal classes, and the second author's proof of the eventual categoricity conjecture in tame AECs with primes \cite{vaseyd, vaseye}. We decided to publish them because it is not completely obvious how to use these tools, and also because we believe that it is worth demonstrating how they can be used to solve such test questions.

Let us outline the proof of Theorem \ref{main-theorem} in more details. We start with $K$, the class of models of our universal $\Ll_{\omega_1, \omega}$ sentence $\psi$. This is a universal class (see Definition \ref{univ-def}). We are further assuming that $\psi$ is categorical in $\aleph_0$ and has one but not too many models in $\aleph_1$. The first step is to show that $K$ is well-behaved in $\aleph_0$: we use machinery of Shelah (Fact \ref{nice-stability-fact}) to build a \emph{good $\aleph_0$-frame}. The second step is to observe that universal classes have a locality property for Galois types called \emph{tameness} (see Definition \ref{tame-def}): in fact Galois types are determined by their finite restrictions (this is due to Will Boney, see Fact \ref{tame-univ}). The third step is the easy observation that in universal classes there is a prime model over every set (see Definition \ref{primet}): take the closure of the set under the functions of an ambient model. The fourth and final step is to use the second author's results on AECs that have a good frame, are tame, and have primes \cite{vaseyd, vaseye}: any such class has arbitrarily large models and further Morley's categoricity theorem holds of such classes.

Note that the above argument only used the structural assumption on the class in the first step (to get the good frame). Once we have a good frame, the result follows because any universal class is tame and has primes. Moreover, the argument to get the good frame works in a much more general setup than universal classes. This is the reason our main theorem can be generalized to Theorem \ref{general-thm}.

  To sum up, any tame AEC with primes which has good behavior in the ``low'' cardinals ($\aleph_0$ and $\aleph_1$) will have good behavior everywhere. If on the other hand it is not clear that the AEC is tame or has primes, Shelah's results \cite{sh87a, sh87b} and the Hart-Shelah example \cite{hash323, untame} tell us that one will need to use higher cardinals (the $\aleph_n$'s) to sort out whether the AEC is well-behaved past $\aleph_\omega$.

This paper was written while the first author was working on a Ph.D.\ under the direction of Rami Grossberg at Carnegie Mellon University and he would like to thank Professor Grossberg for his guidance and assistance in his research in general and in this work in particular. We also would like to thank John Baldwin, Will Boney, Hanif Cheung, and the referees for valuable comments that helped improve the paper.

\section{Preliminaries}

We assume that the reader has some familiarity with the basics of abstract elementary classes, as presented in for example \cite[\S 4-8]{baldwinbook09}. In this section, we recall the main notions that we will use.

The notion of a universal class was studied already in Tarski's \cite{tarski}. Shelah \cite{sh300} was the first to develop classification theory for non-elementary universal classes.

\begin{defin}\label{univ-def}
A class of structures $K$ is a universal class if:
\begin{enumerate}
\item $K$ is a class of $\tau$-structures, for some fixed vocabulary $\tau = \tau(K)$.
\item $K$ is closed under isomorphisms.
\item $K$ is closed under $\subseteq$-increasing chains.
\item If $M \in K$ and $N \subseteq M$, then $N \in K$.  
\end{enumerate}
\end{defin}

The following basic characterization of universal classes is essentially due to Tarski \cite{tarski} (he proved it for finite vocabulary, but the proof generalizes). This will not be used in the present paper.

\begin{fact}[Tarski's presentation theorem]\label{tarski-fact}
  Let $K$ be a class of structures. The following are equivalent:

  \begin{enumerate}
  \item\label{tarski-1} $K$ is a universal class.
  \item\label{tarski-2} $K$ is the class of models of a universal $\Ll_{\infty, \omega}$ theory.
  \end{enumerate}
\end{fact}

In this paper we will use tools of the more general framework of abstract elementary classes:

\begin{defin}[Definition 1.2 in \cite{sh88}]\label{aec-def}
  An \emph{abstract elementary class} (AEC for short) is a pair $\K = (K, \lea)$, where:

  \begin{enumerate}
    \item $K$ is a class of $\tau$-structures, for some fixed vocabulary $\tau = \tau (\K)$. 
    \item $\lea$ is a partial order (that is, a reflexive and transitive relation) on $K$. 
    \item $(K, \lea)$ respects isomorphisms: If $M \lea N$ are in $K$ and $f: N \cong N'$, then $f[M] \lea N'$. In particular (taking $M = N$), $K$ is closed under isomorphisms.
    \item If $M \lea N$, then $M \subseteq N$. 
    \item Coherence: If $M_0, M_1, M_2 \in K$ satisfy $M_0 \lea M_2$, $M_1 \lea M_2$, and $M_0 \subseteq M_1$, then $M_0 \lea M_1$;
    \item Tarski-Vaught axioms: Suppose $\delta$ is a limit ordinal and $\seq{M_i \in K : i < \delta}$ is an increasing chain. Then:

        \begin{enumerate}

            \item $M_\delta := \bigcup_{i < \delta} M_i \in K$ and $M_0 \lea M_\delta$.
            \item\label{smoothness-axiom}Smoothness: If there is some $N \in K$ so that for all $i < \delta$ we have $M_i \lea N$, then we also have $M_\delta \lea N$.

        \end{enumerate}

    \item Löwenheim-Skolem-Tarski axiom: There exists a cardinal $\lambda \ge |\tau(\K)| + \aleph_0$ such that for any $M \in K$ and $A \subseteq |M|$, there is some $M_0 \lea M$ such that $A \subseteq |M_0|$ and $\|M_0\| \le |A| + \lambda$. We write $\LS (\K)$ for the minimal such cardinal.
  \end{enumerate}
\end{defin}
\begin{remark} \
  \begin{enumerate}
    \item When we write $M \lea N$, it is assumed that $M, N \in K$.
    \item We write $\K$ for the pair $(K, \lea)$, and $K$ (no boldface) for the actual class. However we may abuse notation and write for example $M \in \K$ instead of $M \in K$ when there is no danger of confusion.
\item Given $[\lambda, \mu)$ an interval of cardinals (we allow $\mu = \infty$), let $\K_{[\lambda, \mu)} = \{  M \in \K : \| M\| \in [\lambda, \mu) \}$. We write $\K_\lambda$ for $\K_{\{\lambda \}}$ and $\K_{\ge \lambda}$ for $\K_{[\lambda, \infty)}$.
    \item If $K$ is a universal class, then $\K := (K, \subseteq)$ is an AEC with $\LS (\K) = |\tau (K)| + \aleph_0$. Throughout this paper, we think of $K$ as the AEC $\K$ and may write ``$\K$ is a universal class'' instead of ``$K$ is a universal class''.
  \end{enumerate}
\end{remark}

In any AEC $\K$, we can define a semantic notion of type, called Galois or orbital type in the literature (such types were introduced by Shelah in \cite{sh300} but we use the definition from \cite[2.16]{sv-infinitary-stability-afml}).

\begin{defin}\label{gtp-def}
  Let $\K$ be an AEC.
  
  \begin{enumerate}
    \item Let $\K^3$ be the set of triples of the form $(\bb, A, N)$, where $N \in \K$, $A \subseteq |N|$, and $\bb$ is a sequence of elements from $N$. 
    \item For $(\bb_1, A_1, N_1), (\bb_2, A_2, N_2) \in \K^3$, we say $(\bb_1, A_1, N_1)E_{\text{at}} (\bb_2, A_2, N_2)$ if $A := A_1 = A_2$, and there exists $f_\ell : N_\ell \xrightarrow[A]{} N$ such that $f_1 (\bb_1) = f_2 (\bb_2)$.
    \item Note that $E_{\text{at}}$ is a symmetric and reflexive relation on $\K^3$. We let $E$ be the transitive closure of $E_{\text{at}}$.
    \item For $(\bb, A, N) \in \K^3$, let $\gtp_{\K} (\bb / A; N) := [(\bb, A, N)]_E$. We call such an equivalence class a \emph{Galois type} (or just a \emph{type}). Usually, $\K$ will be clear from context and we will omit it.
  \end{enumerate}
\end{defin}

Note that Galois types are defined as the finest notion of type respecting $\K$-embeddings. When $\K$ is an elementary class, $\gtp (\bb / A; M)$ contains the same information as the usual notion of $\Ll_{\omega, \omega}$-syntactic type, but in general the two notions need not coincide \cite{hash323, untame}. We will see shortly (Fact \ref{tame-univ}) that in universal classes the Galois types coincide with the quantifier-free types.

The \emph{length} of $\gtp (\bb / A; M)$ is the length of $\bb$. For $M \in \K$ and $\alpha$ a cardinal, $p$ is a type over $M$ of length $\alpha$ if there is $N \gea M$ and $\bb \in N^\alpha$ such that $p = \gtp(\bb /M, N)$. We write $\gS_{\K}^\alpha (M) = \gS^\alpha (M) = \{\gtp (\bb / M; N) : \bb \in \fct{\alpha}{N}, M \lea N\}$ for the set of types over $M$ of length $\alpha$. When $\alpha = 1$, we just write $\gS (M)$. We define naturally what it means for a type to be realized inside a model, to extend another type, and to take the image of a type by a $\K$-embedding. We call an AEC $\K$ \emph{$\lambda$-stable} if $|\gS (M)| \le \lambda$ for every $M \in \K$ of cardinality $\lambda$.

The notion of a good $\lambda$-frame is introduced in \cite[\S II.2]{shelahaecbook}. As an approximation, the reader can think of the statement ``$\K$ has a good $\lambda$-frame'' as saying ``$\K$ has a model of cardinality $\lambda$, amalgamation in $\lambda$, no maximal models in $\lambda$, joint embedding in $\lambda$, is stable in $\lambda$, and has a superstable-like nonforking notion for types over models of cardinality $\lambda$'' (for a full definition see \cite[3.8]{bova})\footnote{In this paper our frames will always be \emph{type-full}.}.

Tameness is a locality property of Galois types (which may or may not hold), first named by Grossberg and VanDieren in \cite{tamenessone}:

\begin{defin}\label{tame-def}
We say an AEC $\K$ is \emph{$(< \kappa)$-tame} if for any $M \in \K$ and $p \neq q \in \gS(M)$,  there is $A \subseteq |M|$ such that $|A| < \kappa$ and $p\upharpoonright A \neq q\upharpoonright A$. By \emph{$\kappa$-tame} we mean \emph{$(<\kappa^+)$-tame}. If we write \emph{$(< \kappa, \lambda)$-tame} we restrict to $M \in \K_{\lambda}$. We may also talk of tameness for types of \emph{finite length}, which means that we allow $p, q$ above to be in $\gS^{<\omega} (M)$ rather than just in $\gS (M)$ (i.e. they could be types of finite sequences rather than types of singletons).
\end{defin}

The following important fact is due to Will Boney. It appears in print as \cite[3.7]{vaseyd}.

\begin{fact}\label{tame-univ}
If $\K$ is a universal class, then $\K$ is $(< \aleph_{0})$-tame for types of finite length (in fact for types of all lengths). Moreover, Galois types are the same as quantifier-free types.
\end{fact}

The final main concept use in this paper is that of prime models (here over sets of the form $M \cup \{a\}$). The appropriate definition was introduced to AECs by Shelah in \cite[III.3.2]{shelahaecbook}. The definition is what the reader would expect when working inside a fixed monster model, but here we may not have amalgamation, so we have to use Galois types to describe the embedding of the base set.

\begin{defin}\label{primet} Let $\K$ be an AEC.

\begin{itemize}
\item A \emph{prime triple} is $(a,  M, N)$ such that $M \lea N$, $a \in |N|$ and for every $N' \in \K$ and $a' \in |N'|$ such that $\gtp(a/M, N)= \gtp(a'/ M, N')$, there exists $f: N \xrightarrow[M]{} N'$ so that $f(a)=a'$.
\item We say that $\K$ \emph{has primes} if for any $M \in \K$ and every $p \in \gS(M)$, there is a prime triple $(a ,M , N)$ such that $p = \gtp(a/M, N)$.
\end{itemize}
\end{defin}

By taking the closure of $M \cup \{a\}$ under the functions of an ambient model, we obtain \cite[5.3]{vaseyd}:

\begin{fact}\label{prime-univ}
If $\K$ is a universal class, then $\K$ has primes.
\end{fact}

The past two facts show that universal classes are tame and have primes. The next facts show that if we have a good frame in addition to that, then the structure of the frame transfers upward and in fact categoricity can be transferred.

We first give an approximation, due to Boney and the second author \cite[6.9]{bova}, which assumes amalgamation instead of primes (an earlier result is \cite[1.1]{extendingframes}, which assumes tameness for types of length two instead of just length one).

\begin{fact}\label{extending}
  Let $\K$ be an AEC and let $\lambda \ge \LS (\K)$. If $\K$ is $\lambda$-tame, $\K$ has amalgamation and $\K$ has a type-full good $\lambda$-frame, then $\K$ has a type-full good $[\lambda,\infty)$-frame. 
\end{fact}

The second author showed that one could replace amalgamation by primes (in fact a weak version of amalgamation suffices) \cite[4.16]{vaseyd}:

\begin{fact}\label{extending-2}
  Let $\K$ be an AEC and let $\lambda \ge \LS (\K)$. If $\K$ is $\lambda$-tame, has primes, and $\K$ has a type-full good $\lambda$-frame, then $\K_{\ge \lambda}$ has amalgamation. Hence a type-full good $[\lambda, \infty)$-frame by Fact \ref{extending}.
\end{fact}

Finally, the second author used Fact \ref{extending-2} together with the orthogonality calculus of good frames to prove the following categoricity transfer \cite[2.8]{vaseye}:

\begin{fact}\label{categ-fact}
  Let $\K$ be an AEC and let $\lambda \ge \LS (\K)$. Assume that $\K$ is $\lambda$-tame, has primes, is categorical in $\lambda$, and $\K$ has a type-full good $\lambda$-frame. If $\K$ is categorical in \emph{some} $\mu > \lambda$, then $\K$ is categorical in \emph{all} $\mu' > \lambda$.
\end{fact}

To get the good frame, we will use the following result from the study of AECs axiomatized by $\Ll_{\omega_1, \omega}$. It is due to Shelah and already present in some form in \cite{sh48, sh87a} (see also \cite[II.3.4]{shelahaecbook}), but we cite from other sources and sketch some details here for the convenience of the reader. 

\begin{fact}\label{nice-stability-fact}
  Assume $2^{\aleph_0} < 2^{\aleph_1}$. Let $\psi$ be an $\Ll_{\omega_1, \omega}$-sentence. If $1 \le \Ii (\psi, \aleph_1) < 2^{\aleph_1}$, then there exists an AEC $\K$ such that:

  \begin{enumerate}
  \item $\tau (\K) = \tau (\psi)$.
  \item Any model in $\K$ satisfies $\psi$.
  \item For $M, N \in \K$, $M \lea N$ if and only if $M \lee_{\Ll_{\infty, \omega}} N$.
  \item $\K$ is categorical in $\aleph_0$ and has only infinite models.
  \item $\K$ has a type-full good $\aleph_0$-frame.
  \end{enumerate}
\end{fact}

One key of the proof is the following classical consequence of Keisler's omitting type theorem \cite[5.10]{kei70}.

\begin{fact}\label{kei}
  Let $\psi$ be an $\Ll_{\omega_1, \omega}$-sentence and $L^{*}$ be a countable fragment of $\Ll_{\omega_1, \omega}$. If there is a model $M$ of $\psi$ realizing uncountably-many $L^{*}$-types over the empty set, then $\Ii(\psi, \aleph_1)= 2^{\aleph_1}$. 
\end{fact}

Another crucial result of Shelah will be used to obtain amalgamation from few models. See \cite[I.3.8]{shelahaecbook} or \cite[4.3]{grossberg2002} for a proof.

\begin{fact}\label{ap-categ}
Assume $2^{\lambda}< 2 ^{\lambda^{+}}$. Let $\K$ be an AEC.  If $\K$ is categorical $\lambda$ and $\Ii(\K, \lambda^+) < 2^{\lambda^{+}}$, then $\K$ has amalgamation in $\lambda$.
\end{fact} 

We will also use the following fact from {\cite[IV.1.12]{shelahaecbook}} (there it is assumed that $\lambda, \mu > \LS (\K)$ but the proof goes through without this hypothesis).

\begin{fact}\label{elem-fact}
  Let $\K$ be an AEC, let $\lambda \ge \LS (\K)$, and let $\mu$ be an infinite cardinal. If $\K$ is categorical in $\lambda$ and $\lambda = \lambda^{<\mu}$, then for any $M, N \in \K_{\ge \lambda}$, $M \lea N$ implies $M \lee_{\Ll_{\infty, \mu}} N$.
\end{fact}

Finally, we will use \cite[5.8]{shvas-apal}:

\begin{fact}\label{sh-vas}
If $\K$ is categorical in $\aleph_{0}$, has amalgamation and no maximal models in $\aleph_0$, is $(<\aleph_0, \aleph_0)$-tame and is stable in $\aleph_{0}$, then $\K$ has a type-full good $\aleph_{0}$-frame.
\end{fact}

\begin{proof}[Proof sketch for Fact \ref{nice-stability-fact}]
  By \cite[6.3.2]{baldwinbook09}, there is a \emph{complete}  $\Ll_{\omega_1, \omega}$ sentence $\psi_0$ that implies $\psi$ and has a model of cardinality $\aleph_1$. Let $L^\ast$ be a countable fragment containing $\psi_0$ and let $\K := (\Mod (\psi), \preceq_{L^\ast})$.

 Note that $\K$ is an AEC with $\LS (\K) = \aleph_0$, which by completeness of $\psi_0$ is categorical in $\aleph_0$ and has only infinite models. Hence it has joint embedding in $\aleph_0$. Since it has a model of cardinality $\aleph_1$ by assumption, $\K$ also has no maximal models in $\aleph_0$. Moreover, $\K$ has amalgamation in $\aleph_0$ by Fact \ref{ap-categ}. Finally, by Fact \ref{elem-fact} with $\lambda = \mu = \aleph_0$ and since $\K$ has only infinite models, $M \lea N$ if and only if $M \lee_{\Ll_{\infty, \omega}} N$.

  It remains to show that $\K$ has a type-full good $\aleph_0$-frame. We first show:

  \underline{Claim}: Let $M \in \K_{\aleph_0}$. If $\seq{p_i : i < \omega_1}$ are Galois types over $M$, then there exists $i < j < \omega_1$ such that $p_i \rest A = p_j \rest A$ for all finite $A \subseteq |M|$.

  \underline{Proof of Claim}: By amalgamation in $\aleph_0$, we can find an uncountable model $N$ extending $M$ such that all the $p_i$'s are realized inside $N$. Say $p_i = \gtp (\ba_i / M; N)$. For $A \subseteq |M|$, let $\tau_A$ denote $\tau (\K) \cup \{c_a \mid a \in A\}$, where the $c_a$'s are new constant symbols. Whenever $M \lea N' \lea N$, let $N_A'$ denote the expansion of $N'$ to $\tau_A$ with $c_a^{N'} = a$. Observe that whenever $M \lea N' \lea N$ and $A \subseteq |M|$ is finite, then, since $\lea = \lee_{\Ll_{\infty, \omega}}$, we have that $M_A \lee_{\Ll_{\infty, \omega} (\tau_A)} N_A' \lee_{\Ll_{\infty, \omega} (\tau_A)} N_A$. 

  Let $L^{\ast\ast}$ be a countable fragment of $\Ll_{\omega_1, \omega}$ extending $L^{\ast}$ and containing Scott sentences of $M_A$ for all $A \subseteq |M|$ finite. We now apply Fact \ref{kei} to the following sentence:

  $$
  \bigwedge_{n\in \omega}\{\phi (c_{a_0}, \ldots c_{a_{n - 1}}) : \phi \in L^{\ast\ast}, a_0, \ldots, a_{n - 1} \in |M|, M \models \phi[a_0, \ldots, a_{n - 1}]\}
  $$

  Note that the models of this sentence are essentially the extensions of $M$. Moreover $2^{\aleph_0} < 2^{\aleph_1}$ implies that the sentence still has few models in $\aleph_1$. Thus Fact \ref{kei} indeed applies and we get in particular that there must exist $i < j$ such that $\tp_{L^{**}} (\ba_i / \emptyset; N_{|M|}) = \tp_{L^{**}} (\ba_j / \emptyset; N_{|M|})$. Now fix $N' \lea N$ countable containing $M$ and $\ba_i \ba_j$. Also fix $A \subseteq |M|$ finite. Since $M_{A} \lee_{\Ll_{\infty, \omega} (\tau_A)} N'_{A}$, there exists an isomorphism $f: N' \cong_A M$. Let $\bb_i := f (\ba_i)$, $\bb_j := f (\ba_j)$. 
 By equality of the types $(N'_A, \{ c_{b^i_k}^{N'_A} = a^i_k \}_{k<n}) \equiv_{L^{**} \rest \tau_{A \bb_i}} (N'_A, \{ c_{b^i_k}^{N'_A} = a^j_k \}_{k<n})$, hence $(M_A, \{ c_{b^i_k}^{M_A} = b^i_k \}_{k<n}) \equiv_{L^{**} \rest \tau_{A \bb_i}} (M_A, \{ c_{b^i_k}^{M_A} = b^j_k \}_{k<n})$. Since $L^{**}$ includes all the relevant Scott sentences, this means that there exists an automorphism $g$ of $M$ sending $\bb_i$ to $\bb_j$ and fixing $A$. Composing maps, we obtain an automorphism of $N'$ fixing $A$ and sending $\ba_i$ to $\ba_j$. Thus $p_i \rest A = p_j \rest A$, as desired. $\dagger_{\text{Claim}}$

 Combining the Claim with \cite[3.12]{finitary-aec}, we get that $\K$ is stable in $\aleph_0$ and is $(<\aleph_0, \aleph_0)$-tame for types of finite length. Therefore by Fact \ref{sh-vas}, $\K$ has a type-full good $\aleph_0$-frame.
\end{proof}

\section{Main results} 

In this section we prove the main theorems of this paper. We start by applying Fact \ref{elem-fact} to a universal class categorical in $\aleph_0$:

\begin{lemma}\label{univ-defin}
  Let $\K$ be a universal class in a countable vocabulary. If $\K$ is categorical in $\aleph_0$, then  for $M, N \in \K_{\geq \aleph_0}$, $M \subseteq N$ if and only if $M \lee_{\Ll_{\infty, \omega}} N$. Moreover, $\K_{\geq \aleph_0}$ is the class of models of an $\Ll_{\omega_1, \omega}$-sentence. 
\end{lemma}
\begin{proof}
  Use Fact \ref{elem-fact} with $\lambda = \mu = \aleph_0$ and recall that $\lea$ is just the substructure relation. For the moreover part, take the Scott sentence of a countable model. 
\end{proof}

Applying Fact \ref{nice-stability-fact}, we get directly:

\begin{cor}\label{aleph0-stable}
  Assume $2^{\aleph_0} < 2^{\aleph_1}$. Let $\K$ be a universal class in a countable vocabulary. If $\K$ is categorical in $\aleph_0$ and $1 \le \Ii (\K, \aleph_1) < 2^{\aleph_1}$, then $\K$ has a type-full good $\aleph_0$-frame.
\end{cor}
\begin{proof}
  By Lemma \ref{univ-defin}, $\K_{\geq \aleph_0}$ is axiomatized by an $\Ll_{\omega_1, \omega}$ sentence and the ordering on $\K_{\geq \aleph_0}$ coincides with $\lee_{\Ll_{\infty, \omega}}$. Since $\K$ is already categorical, $\K_{\geq \aleph_0}$ is equal to the class given by Fact \ref{nice-stability-fact}, so $\K$ has a type-full good $\aleph_0$-frame.
\end{proof}

We obtain one of our main theorems:

\begin{theorem}\label{main-theorem}
 Assume $2^{\aleph_0} < 2^{\aleph_1}$. Let $\K$ be a universal class in a countable vocabulary. If  $\K$ is categorical in $\aleph_{0}$ and $1 \leq \Ii(\K, \aleph_1) < 2^{\aleph_1}$, then:

\begin{enumerate}
\item $\K$ has arbitrarily large models.
\item If $\K$ is categorical in some uncountable cardinal then $\K$ is categorical in all uncountable cardinals.
\end{enumerate} 
\end{theorem}
\begin{proof}
By Corollary \ref{aleph0-stable}, $\K$ has a type-full good $\aleph_0$-frame. By facts \ref{tame-univ} and \ref{prime-univ}, $\K$ is $\aleph_0$-tame and has primes. Therefore Fact \ref{extending-2} yields (1) and Fact \ref{categ-fact} yields (2).
\end{proof}

Observe that the only place where we used the hypotheses ``$2^{\aleph_0} < 2^{\aleph_1}$ and $1 \le \Ii (\K, \aleph_1) < 2^{\aleph_1}$'' was to derive amalgamation and stability. Thus the conclusion of Theorem \ref{main-theorem} also holds in ZFC if we assume that $\K$ is universal, $\aleph_0$-categorical, has amalgamation and no maximal models in $\aleph_0$, and is stable in $\aleph_0$ (using Fact \ref{sh-vas} to get the good frame).

We can also replace the assumption of categoricity in $\aleph_0$ by categoricity in $\aleph_1$. To see this, we will use the following local version of Facts \ref{extending}, \ref{extending-2}, \ref{categ-fact}.

\begin{fact}\label{extend-cat}
Let $\K$ be an AEC and $\lambda \geq \LS(\K)$. If $\K$ has a type-full good $\lambda$-frame, is categorical in $\lambda$ and $\lambda^{+}$, is $(\lambda, \lambda^+)$-tame, and $\K_{\lambda^+}$ has primes, then $\K$ has a type-full good $\lambda^+$-frame and is categorical in $\lambda^{++}$.
\end{fact}
\begin{proof}
  The proof of Fact \ref{extending-2} is local, so $\K$ has a good $\lambda^+$-frame. That $\K$ is $\lambda^{++}$-categorical follows from \cite[6.14]{downward-categ-tame-apal}.
\end{proof}
\begin{theorem}\label{categ-upward}
 Assume $2^{\aleph_0} < 2^{\aleph_1}$. Let $\psi$ be a universal $\Ll_{\omega_1, \omega}$ sentence. If $\psi$ is categorical in $\aleph_1$, then it is categorical in all uncountable cardinals.
\end{theorem}
\begin{proof}
 Let $\K$ be the class of models of $\psi$. Let $\K^\ast$ be the class obtained in  Fact \ref{nice-stability-fact}. Note that since $\K_{\aleph_1}^\ast \neq \emptyset$ (by the existence of the good $\aleph_0$-frame), $\K^\ast \subseteq \K$ and $\K$ is categorical in $\aleph_1$, $\K^\ast$ is also categorical in $\aleph_1$. Moreover $\K_{\aleph_1}^\ast = \K_{\aleph_1}$, because by Fact \ref{elem-fact} with $\lambda = \aleph_1$ and $\mu = \aleph_0$ for $M, N \in \K_{\aleph_1}$, $M \subseteq N$ if and only if $M \lee_{\Ll_{\infty, \omega}} N$. Since the behavior of an AEC is determined by its behavior in the Löwenheim-Skolem-Tarski number, $\K_{\ge \aleph_1}^\ast = \K_{\ge \aleph_1}$.

Now $\K^\ast$ has a type-full good $\aleph_0$-frame and since $\K$ is a universal class, $\K$ is $(\aleph_0, \aleph_1)$-tame. Since $\K_{\ge \aleph_1}^\ast = \K_{\ge \aleph_1}$, one can check that $\K^\ast$ is also $(\aleph_0, \aleph_1)$-tame. Furthermore, since $\K_{\aleph_1}$ has primes and $\K_{\aleph_1}^\ast = \K_{\aleph_1}$, $\K^{\ast}_{\aleph_1}$ also has primes. By Fact \ref{extend-cat}, $\K^\ast$ has a type-full good $\aleph_1$-frame and is categorical in $\aleph_2$. But this means that $\K_{\ge \aleph_1}$ has a type-full good $\aleph_1$-frame and is categorical in $\aleph_2$, so we can now apply Fact \ref{categ-fact}, to $\K_{\ge \aleph_1}$ to get the result.
\end{proof}

\section{Open questions and generalizations}

The following variation on an example of Morley shows that for every countable ordinal $\alpha$ there are universal classes with models only up to size $\beth_{\alpha}$.

\begin{example}\label{morley-example}
  Fix $\alpha  < \omega_1$. Let $\tau$ be a vocabulary consisting of unary predicates $\seq{P_i : i \le \alpha}$, a binary relation $E$ and a binary function $f$. Let $K$ be the class of $\tau$-structures $M$ such that:

  \begin{enumerate}
  \item $P_i^M \subseteq P_j^M$ for all $i < j < \alpha$.
  \item $P_0^M = \emptyset$.
  \item $|M| = P_\alpha^M$.
  \item $P_i^M = \bigcup_{j < i} P_j^M$ for $i$ limit.
  \item $x E^M y$ implies $x \in P_i^M$ and $y \in P_j^M$ for some $i < j< \alpha$.
  \item For any $i < \alpha$ and any two distinct $y_1, y_2 \in P_i^M$, $x:= f (y_1, y_2)$ satisfies:

    $$
    \left(x E y_1 \land \neg (x E y_2)\right) \lor \left(\neg (x E y_1) \land x E y_2\right)
    $$
  \end{enumerate}

  Then $K$ is a universal class in a countable vocabulary with amalgamation, joint embedding, and a model of cardinality $\beth_{\alpha} (0)$ but no models of cardinality $\beth_{\alpha} (0)^+$.\footnote{For a (possibly finite) cardinal $\mu$ and an ordinal $\alpha$, $\beth_\alpha (\mu)$ is defined inductively by $\beth_{0} (\mu) = \mu$, $\beth_{\beta + 1} (\mu) = 2^{\beth_\beta (\mu)}$, and $\beth_\delta (\mu) = \sup_{\beta < \delta} \beth_\beta (\mu)$ for $\delta$ limit.} Taking the disjoint union of $K$ with the class of $\mathbb{Q}$-vector spaces, we obtain (when $\alpha \ge \omega$) a universal class in a countable vocabulary which is categorical in an infinite cardinal $\lambda$ exactly when $\lambda > \beth_{\alpha} (0)$.
\end{example}

This shows that some conditions on the class are necessary to derive arbitrarily large models. However it is not clear to us that Theorem \ref{main-theorem} is optimal. Indeed it is not clear to us that the hypotheses on $\aleph_1$ are necessary (see Baldwin-Lachlan \cite{baldwin-lachlan} for a positive result when $K$ is axiomatized by a Horn theory):

\begin{question}
  If $K$ is a universal class categorical in $\aleph_0$ with a model in $\aleph_1$, must it be categorical in $\aleph_1$?
\end{question}

It would also be really nice to have a proof of Theorem \ref{main-theorem} in $ZFC$, so it is natural to ask the following question.

\begin{question}
Can we drop the hypothesis $2^{\aleph_{0}} < 2 ^{\aleph_{1}}$ from Theorem \ref{main-theorem}? Can it be dropped if we add more categoricity assumptions?
\end{question} 

Shelah \cite[\S I.6]{shelahaecbook} has given an example of an analytic AEC which under Martin's axiom is categorical in $\aleph_0$ and $\aleph_1$ yet does not have amalgamation in $\aleph_0$. It seems however plausible that there are no such examples which are universal classes.

We end this paper with a generalization of Theorem \ref{main-theorem}. The key is that we have not used the full strength of the universal assumption: all we used was tameness, having primes, and some definability. Using harder results of Shelah, Theorems \ref{main-theorem} and \ref{categ-upward} generalize to:

\begin{thm}\label{general-thm}
  Assume $2^{\aleph_0} < 2^{\aleph_1}$. Let $\K$ be an AEC with $\LS (\K) = \aleph_0$. Assume that $\K$ has primes, is $\aleph_0$-tame, and is $\PC_{\aleph_0}$ (see \cite[I.1.4]{shelahaecbook}).

  \begin{enumerate}
  \item If $\K$ is categorical in $\aleph_0$ and $1 \le \Ii (\K, \aleph_1) < 2^{\aleph_1}$, then $\K$ has arbitrarily large models and categoricity in \emph{some} uncountable cardinal implies categoricity in \emph{all} uncountable cardinals.
  \item If $\K$ is categorical in $\aleph_1$, then $\K$ is categorical in all uncountable cardinals.
  \end{enumerate}
\end{thm}
\begin{proof}
  As in the proof of Theorems \ref{main-theorem}, \ref{categ-upward} but using \cite[I.3.10]{shelahaecbook} to derive an $\aleph_0$-categorical subclass and \cite[II.3.4]{shelahaecbook} to derive the good $\aleph_0$-frame (actually in this case we only obtain a semi-good $\aleph_0$-frame with conjugation, see \cite[2.3.10]{jrsh875}, but this suffices for the proof).
\end{proof}


\end{document}